\documentclass[11pt]{amsart}
\usepackage{amscd, amssymb, pgf, tikz, hyperref}

\newcommand{\field}[1]{\mathbb{#1}}

\newcommand{\NN}{\field{N}}

\newcommand{\TT}{\field{T}}
\newcommand{\ZZ}{\field{Z}}

\newcommand{\Ff}{\mathcal{F}}
\newcommand{\Gg}{\mathcal{G}}
\newcommand{\Hh}{\mathcal{H}}
\newcommand{\Kk}{\mathcal{K}}
\newcommand{\Ll}{\mathcal{L}}
\newcommand{\Mm}{\mathcal{M}}

\newcommand{\Oo}{\mathcal{O}}

\newcommand{\Ss}{\mathcal{S}}

\newcommand{\Xx}{\mathcal{X}}

\newcommand{\la}{\langle}
\newcommand{\ra}{\rangle}

\newcommand{\ds}{\displaystyle}

\newtheorem{thm}{Theorem}[section]

\newtheorem{prop}[thm]{Proposition}

\theoremstyle{definition}
\newtheorem{dfn}[thm]{Definition}

\theoremstyle{remark}
\newtheorem{rmk}[thm]{Remark}

\newtheorem{example}[thm]{Example}

\newtheorem*{examples*}{Examples}

\numberwithin{equation}{subsection}
\numberwithin{equation}{subsection}

\title{ On  groupoids and $C^*$-algebras from  self-similar actions}

\author{Valentin Deaconu}
\address{Valentin Deaconu \\ Department of Mathematics \& Statistics\\ University
of Nevada\\ Reno NV 89557-0084\\ USA} \email{vdeaconu@unr.edu}



\begin{document}
\begin{abstract}
Given a self-similar groupoid action $(G,E)$ on the path space of a finite graph, we study the  associated Exel-Pardo \' etale groupoid $\Gg(G,E)$ and its $C^*$-algebra $C^*(G,E)$. 
We review some facts about groupoid actions, skew products and semi-direct products and generalize a result of Renault about similarity of  groupoids which resembles Takai duality. We also describe a general strategy to compute the $K$-theory of $C^*(G,E)$ and the homology of $\Gg(G,E)$ in certain cases and  illustrate with an example. 
\end{abstract}
\maketitle
\section{introduction}

\bigskip

\bigskip

Informally, a self-similar action is given by isomorphisms between parts of an object (think fractals or Julia sets) at different scales. Self-similar actions were studied intensely after exotic examples of groups acting on rooted trees and generated by finite automata,  like  infinite residually finite torsion groups, and groups of intermediate growth were constructed by Grigorchuk in the 1980's. Using the Pimsner construction from a $C^*$-correspondence, Nekrashevych introduced the $C^*$-algebras associated  with self-similar group actions in \cite{N, N1}, where important results about their structure and their $K$-theory were obtained. Motivated by the construction of all Kirchberg algebras in the UCT class using topological graphs given by Katsura, in \cite {EP} Exel and Pardo introduced self-similar group actions on graphs  and realized their $C^*$-algebras as groupoid $C^*$-algebras.

In this paper, we are interested in self-similar actions of groupoids $G$ on the path space of finite directed graphs $E$ as introduced and studied in \cite{LRRW}, where the main  goal was to find KMS states on some resulting dynamical systems. We generalize certain results of Exel and Pardo, in particular we define a groupoid $\Gg(G,E)$  and discuss the structure of the $C^*$-algebra $C^*(G,E)$, defined as a Cuntz-Pimsner algebra of a $C^*$-correspondence over $C^*(G)$. The $C^*$-algebra  $C^*(G,E)$ has a natural gauge action and contains copies of $C^*(E)$ and $C^*(G)$. In general, its structure is rather intricate; in a particular case, $C^*(G,E)\cong C^*(E)\rtimes G$. 

We begin with a review of \' etale groupoid homology, then we recall  some facts about groupoid actions, skew products and semi-direct products and generalize a result of Renault about similarity of  groupoids in the spirit of Takai duality. We also describe a general strategy to compute the $K$-theory of $C^*(G,E)$ and the homology of $\Gg(G,E)$ in certain cases. We illustrate with an example. 

We expect that many of our results will be true for self-similar actions of groupoids  on the path space of infinite graphs. For the  case when $G$ is a group, see \cite{EPS, L}.

\bigskip

\bigskip

\section{Homology of \'etale groupoids}

\bigskip

A groupoid $G$ is a  small category with inverses. The set of objects is denoted by $G^{(0)}$. We will use $d$ and $t$ for the domain and terminus maps $d,t:G \to G^{(0)}$ to distinguish them from the range and source maps $r,s$ on directed graphs. For $u,v \in G^{(0)}$, we write \[G_u =\{g\in G: d(g)=u\},\;\;G^v =\{g\in G: t(g)=v\},\;\; G_u^v=G_u\cap G^v.\]
The set of composable pairs is denoted $G^{(2)}$.

An \'etale groupoid is a topological groupoid where the terminus map $t$ (and necessarily the
domain map  $d$) is a local homeomorphism (as a map from $G$ to $G$). The unit space $G^{(0)}$ of an \'etale groupoid is always an open subset of $G$.

\begin{dfn} Let $G$ be an \' etale groupoid. A bisection is an open subset $U\subseteq  G$ such that $d$ and $t$ are both injective when restricted to $U$.
\end{dfn}

Two units $x,y \in G^{(0)}$ belong to the same $G$-orbit if there exists $g \in G$ such that $d(g) = x$ and $t(g) = y$. We denote by orb$_G(x)$ the $G$-orbit of $x$. When every $G$-orbit is dense in $G^{(0)}$, the groupoid $G$ is called minimal. An open set $V\subseteq G^{(0)}$ is called $G$-full if for every $x \in G^{(0)}$ we have orb$_G (x) \cap V \neq \emptyset$.
We denote by $G_V$ the subgroupoid $\{g \in G\; | \; d(g), t(g) \in V\}$, called the restriction of $G$ to $V$. When $G$ is \'etale, the restriction $G_V$ is an open \'etale subgroupoid with unit space $V$.

The isotropy group of a unit $x\in G^{(0)}$ is the group \[G_x^x :=\{g\in G\; | \; d(g)=t(g)=x\},\] and the isotropy bundle is
\[G' :=\{g\in G\; | \; d(g)=t(g)\}= \bigcup_{x\in G^{(0)}} G_x^x.\]
A groupoid $G$ is said to be principal if all isotropy groups are trivial, or equivalently, $G' = G^{(0)}$. We say that $G$ is effective if the interior of $G'$ equals $G^{(0)}$.

\begin{dfn} A groupoid $G$ is elementary if it is compact and principal. A groupoid $G$ is an $AF$ groupoid if there exists an ascending chain of open elementary subgroupoids $G_1\subseteq  G_2\subseteq ...\subseteq G$ such that $G = \bigcup_{i=1}^\infty G_i$. A groupoid $G$ is ample if it is \'etale and $G^{(0)}$ is zero-dimensional; equivalently, $G$ is ample if it has a basis of compact open bisections.
\end{dfn}

We recall now the definion of homology of \'etale groupoids which was introduced by Crainic and Moerdijk in \cite{CM}. Let $A$ be an Abelian group and
let $\pi: X \to Y$ be a local homeomorphism between two locally compact Hausdorff spaces. Given any $f \in C_c(X, A)$ we define
\[\pi_*(f)(y):= \sum_{\pi(x)=y} f(x).\]

It follows that $\pi_*(f)\in C_c(Y,A)$.
Given an \' etale groupoid $G$, let $G^{(1)}=G$ and for $n\ge 2$ let $G^{(n)}$ be the space of composable strings
of $n$ elements in $G$ with the product topology. For $n\ge 2$ and  $ i = 0,...,n$, we let $\partial_i : G^{(n)} \to G^{(n-1)}$ be the face maps defined by

\[\partial_i(g_1,g_2,...,g_n)=\begin{cases}(g_2,g_3,...,g_n)&\;\text{if}\; i=0,\\
 (g_1,...,g_ig_{i+1},...,g_n) &\;\text{if} \;1\le i\le n-1,\\
(g_1,g_2,...,g_{n-1})& \;\text{if}\; i = n.\end{cases}\]

We define the homomorphisms $\delta_n : C_c(G^{(n)}, A) \to C_c(G^{(n-1)}, A)$ given by 
\[\delta_1=d_*-t_*,\;\; \delta_n=\sum_{i=0}^n(-1)^i\partial_{i*}\; \text{for}\; n\ge 2.\]
It can be verified that $\delta_n\circ\delta_{n+1}=0$ for all $n\ge 1$.

The homology groups $H_n(G, A)$ are by definition  the homology groups of the chain complex $C_c(G^{(*)},A)$ given by
\[
0\stackrel{\delta_0}{\longleftarrow} C_c(G^{(0)},A)\stackrel{\delta_1}{\longleftarrow}C_c(G^{(1)},A)\stackrel{\delta_2}{\longleftarrow}C_c(G^{(2)},A)\longleftarrow\cdots,\]
i.e. $H_n(G,A)=\ker \delta_n/\text{im } \delta_{n+1}$, where $\delta_0=0$.
If $A=\ZZ$, we write $H_n(G)$ for $H_n(G,\ZZ)$.

The following HK-conjecture of Matui states that the homology of an \'etale  groupoid
refines the $K$-theory of the reduced groupoid $C^*$-algebra. 
 Let $G$ be a minimal effective ample Hausdorff groupoid with compact unit space. Then

\[K_i(C_r^*(G)) \cong \bigoplus_{n=0}^\infty H_{2n+i}(G), \;\text{for}\;  i = 0, 1.\]

Recently,  this conjecture was the source of intense research. It was confirmed for several groupoids like $AF$-groupoids, transformation groupoids of Cantor minimal systems, groupoids of shifts of finite type and products of groupoids of shifts of finite type, see \cite{M}. The homology of ample Hausdorff groupoids was investigated in \cite{FKPS}, with emphasis on the  Renault-Deaconu groupoids associated to $k$ pairwise-commuting local homeomorphisms of a zero-dimensional space. It was shown that the homology of $k$-graph groupoids can be computed in terms of the adjacency matrices, using spectral sequences and a chain complex developed by Evans in \cite{E}. The HK-conjecture was  also confirmed for groupoids on one-dimensional solenoids in \cite{Y}. Recently, counterexamples to the HK-conjecture of Matui were found by  Scarparo in \cite{S}  and by Ortega and Sanchez in \cite{OS}.

\bigskip

\section{Groupoid actions and similarity}

\bigskip

We recall  the concept of a groupoid action on another groupoid from \cite{AR}, page 122 and from \cite{De}. 
\begin{dfn} \label{ga}
A topological groupoid $G$ acts (on the right) on another topological groupoid $H$ if there are a continuous open surjection $p: H\to G^{(0)}$ and a continuous map $H\ast G\to H$, write $(h,g)\mapsto h\cdot g$ where
  \[H\ast G=\{(h,g)\in H\times G \mid t(g)=p(h)\}\] 
   such that
   
   \medskip
   
   i)  $p(h\cdot g)=d(g)$  for all $(h,g)\in H\ast G$,
   
   \medskip
 
 ii) $(h,g_1)\in H\ast G$ and $(g_1, g_2)\in G^{(2)}$ implies that $(h, g_1g_2)\in H\ast G$ and \[ h\cdot (g_1g_2)=(h\cdot g_1)\cdot g_2,\]
 
 \medskip
 
 iii) $(h_1,h_2)\in H^{(2)}$ and $( h_1h_2, g)\in H\ast G$ implies $(h_1, g), (h_2, g)\in H\ast G$ and \[(h_1h_2)\cdot g=(h_1\cdot g)(h_2\cdot g),\]
 
 \medskip
 
 iv) $ h\cdot p(h)=h$ for all $h\in H$.
 
 \noindent The action is called free if $h\cdot g=h$ implies $g=p(h)$ and transitive if for all $h_1,h_2\in H$ there is $g\in G$ with $h_2=h_1\cdot g$.
 \end{dfn}
  Note that if $G$ acts on $H$ on the right, we can define a left action of $G$ on $H$ by taking $g\cdot h:= h\cdot g^{-1}$ and viceversa.
 \begin{example}
Given $G$ a topological groupoid, a $G$-module in \cite{T} is a topological groupoid $A$ with domain and terminus maps equal to $p:A\to G^{(0)}$ such that $A_x^x$ is an abelian group for all $x\in G^{(0)}$, $G$ acts on $A$ as a space and such that for each $g\in G$ the action map $\alpha_g:A_{d(g)}\to A_{t(g)}$ is a group homomorphism. In particular, $A$ can be a trivial group bundle $G^{(0)}\times D$ for $D$ an abelian group and $\alpha_g=id_D$ for all $g\in G$.
\end{example}

 \begin{rmk}
If $G$ acts on the groupoid $H$, then  $G$  acts on the unit space $H^{(0)}$ using the restriction $p_0:=p|_{H^{(0)}}:H^{(0)}\to G^{(0)}$ and we have $p=p_0\circ t=p_0\circ d$, where $d, t:H\to H^{(0)}$. In particular,
\[p(h^{-1})=p_0(t(h^{-1}))=p_0(d(h))=p(h).\]
Using the fact that $h=hd(h)=t(h)h$, it follows that \[h\cdot g=( h\cdot g)( d(h)\cdot g)=( t(h)\cdot g)( h\cdot g),\]
so  we deduce that  $d( h\cdot g)=d(h)\cdot g$ and $t( h\cdot g)= t(h)\cdot g$.
\end{rmk}
 \begin{dfn}
If $G$ acts on $H$, then the semi-direct product groupoid $H\rtimes G$, also called the action groupoid,  is defined as follows. As a set, \[ H\rtimes G=H\ast G=\{(h,g)\in H\times G \mid t(g)=p(h)\}\] and  the multiplication is given by
\[(h,g)(h'\cdot g,g')=(hh', gg'),\]
when $t(g')=d(g)$ and $d(h)=t(h')$. 

\end{dfn}
In a semi-direct product, the inverse is given by \[(h,g)^{-1}=(h^{-1}\cdot g, g^{-1})\]
and we get
\[(h,g)^{-1}(h,g)=(h^{-1}\cdot g, g^{-1})(h,g)=((h^{-1}\cdot g)(h\cdot g), d(g))=(d(h)\cdot g, d(g)),\]
\[(h,g)(h,g)^{-1}=(h,g)(h^{-1}\cdot g, g^{-1})=(t(h),t(g)).\] 
Since $d(g)=p(d(h)\cdot g)$ and $ t(g)=p( t(h))$, the unit space of $H\rtimes G$ can be identified with $H^{(0)}$ and then we make identifications 
\[d(h,g)\equiv d(h)\cdot g,\;\; t(g,h)\equiv t(h).\]  There is a groupoid homomorphism \[\pi: H\rtimes G\to G,\; \pi(h,g)=g\] with kernel $\pi^{-1}(G^{(0)})=\{(h,p(h))\;\mid\; h\in H\}$ isomorphic to $H$.

\begin{rmk}
The notion of groupoid action on another groupoid   includes the action of a groupoid on a space and the action of a group on another group by automorphisms. A particular situation is when $G_1, G_2$ are groupoids and $Z$ is a $(G_1,G_2)$-space, i.e. $Z$ is a left $G_1$-space, a right $G_2$-space and the actions commute. Then 
$G_1\ltimes Z$ is a right $G_2$-groupoid and $Z\rtimes G_2$ is a left $G_1$-groupoid.
\end{rmk}

\begin{dfn}
Suppose now that $G,\Gamma$ are \'etale groupoids and that $\rho:G\to \Gamma$ is a groupoid homomorphism, also called a cocycle. The skew product groupoid $G\times_{\rho}\Gamma$ is defined as the set of pairs $(g,\gamma)\in G\times \Gamma$ such that $(\gamma,\rho(g))\in \Gamma^{(2)}$ with multiplication
\[(g,\gamma)(g',\gamma\rho(g))=(gg',\gamma)\;\text{if}\;(g,g')\in G^{(2)}\]
and inverse
\[(g,\gamma)^{-1}=(g^{-1},\gamma\rho(g)).\]
\end{dfn}
In a skew product we have $d(g,\gamma)=(d(g),\gamma\rho(g))$ and $t(g,\gamma)=(t(g),\gamma)$. Its unit space is \[G^{(0)}\ast \Gamma=\{(u,\gamma)\in G^{(0)}\times \Gamma: \rho(u)=d(\gamma)\}.\]

In particular, if $G,H$ are \'etale groupoids and $G$ acts on $H$ on the right, for the groupoid homomorphism $\pi: H\rtimes G\to G,\; \pi(h,g)=g$  we can form the skew product $(H\rtimes G)\times_\pi G$ made of triples $(h,g,g')\in H\times G\times G$ such that $p(h)=t(g)$ and $(g',g)\in G^{(2)}$,  with unit space $H^{(0)}\ast G$ and operations 
\[(h,g,g')(h', g'',g'g)=(h(h'\cdot g^{-1}),gg'', g') ,\]
\[(h,g,g')^{-1}=(h^{-1}\cdot g,g^{-1}, g'g).\]
\begin{rmk}
Given a groupoid homomorphism $\rho:G\to \Gamma$,   there is a left action $\hat{\rho}$ of $\Gamma$ on the skew product  $G\times_{\rho}\Gamma$ given by \[\gamma'\cdot(g,\gamma)=(g,\gamma'\gamma).\] 
\end{rmk}
\begin{proof}
We check all the properties defining a groupoid action. First, we define the continuous open map \[p:G\times_{\rho}\Gamma\to \Gamma^{(0)},\; p(g,\gamma)=t(\gamma)\] and note that $d(\gamma')=t(\gamma)=p(g,\gamma)$.  Now $p(g,\gamma'\gamma)=t(\gamma')$ and if $(\gamma_1, \gamma_2)\in \Gamma^{(2)}$, then \[(\gamma_1\gamma_2)\cdot(g,\gamma)=(g, \gamma_1\gamma_2\gamma)=\gamma_1\cdot(\gamma_2\cdot(g,\gamma)).\] Also, if $(g,\gamma), (g', \gamma\rho(g))\in G\times_\rho\Gamma$ are composable with product $(gg', \gamma)$, then
\[\gamma'\cdot(gg',\gamma)=(gg', \gamma'\gamma)=(g, \gamma'\gamma)(g', \gamma'\gamma\rho(g))=(\gamma'\cdot(g,\gamma))(\gamma'\cdot(g', \gamma\rho(g))).\]
The last condition to check is $t(\gamma)\cdot(g,\gamma)=(g,\gamma)$, which is obvious.
\end{proof}
We define  a right action of $\Gamma$  on $G\times_{\rho}\Gamma$ by $(g,\gamma)\cdot \gamma'=(g,\gamma'^{-1}\gamma)$, and we form the semi-direct product $(G\times_\rho \Gamma)\rtimes\;\Gamma$ made of triples $(g,\gamma,\gamma')\in G\times \Gamma\times \Gamma$ such that $(\gamma, \rho(g))\in \Gamma^{(2)}$ and $t(\gamma')=p(g,\gamma)=t(\gamma)$, with unit space $G^{(0)}\ast\Gamma$ and operations
\[(g,\gamma,\gamma')(g', \gamma'^{-1}\gamma\rho(g),\gamma'')=(gg', \gamma,\gamma'\gamma'') ,\]
\[(g, \gamma, \gamma')^{-1}=((g^{-1},\gamma\rho(g))\cdot \gamma', \gamma'^{-1} )=(g^{-1}, \gamma'^{-1}\gamma\rho(g),\gamma'^{-1}).\] 

For the next result, which resembles Takai duality, see  Definition 1.3 in \cite{Re}, Proposition 3.7 in \cite{M}  and Definition 3.1 in \cite{FKPS}.

\begin{thm}
Let $G, H, \Gamma$ be \'etale groupoids such that $G$ acts on $H$ and such that $\rho:G\to \Gamma$ is a groupoid homomorphism. Then, using  the above notation, $(H\rtimes G)\times_\pi G$ is similar to $H$ and $(G\times_\rho \Gamma)\rtimes \Gamma$ is similar to $G$. 
\end{thm}	
\begin{proof}
Recall  that two (continuous) groupoid homomorphisms $\phi_1,\phi_2:G_1\to G_2$ are similar if there is  a continuous function $\theta:G_1^{(0)}\to G_2$ such that 
\[\theta(t(g))\phi_1(g)=\phi_2(g)\theta(d(g))\]
for all $g\in G_1$. Two topological groupoids $G_1, G_2$ are similar if there exist continuous homomorphisms $\phi:G_1\to G_2$ and $\psi:G_2\to G_1$ such that $\psi\circ \phi$ is similar to $id_{G_1}$  and $\phi\circ \psi$ is similar to $id_{G_2}$.

 To show that $(H\rtimes G)\times_\pi G$ is similar to $H$, we define 
 \[\phi: (H\rtimes G)\times_\pi G\to H,\; \phi(h,g,g')= h\cdot g'^{-1}, \]
 \[\psi:H\to (H\rtimes G)\times_\pi G,\; \psi(h)=(h,p(h),p(h))\]
 and 
 \[\theta:H^{(0)}\ast G\to (H\rtimes G)\times_\pi G,\;\theta(u,g)=(u\cdot g^{-1}, g, p(u)).\] 
 We  check that \[(\phi\circ\psi)(h)=\phi(h,p(h),p(h))=h\cdot p(h)^{-1}=h\] and  that \[(*)\;\; \theta[t(h,g,g')](h,g,g')=(\psi\circ\phi)(h,g,g')\theta[d(h,g,g')].\]
 We have
 
 \[t(h,g,g')=(h,g,g')(h,g,g')^{-1}=(h,g,g')(h^{-1}\cdot g, g^{-1},g'g)=\]
 \[=(h(h^{-1}\cdot g\cdot g^{-1}),gg^{-1},g')=(t(h),t(g),g')\equiv (t(h),g'),\]
 and 
 \[\theta(t(h),g')=(t(h)\cdot g'^{-1}, g',p(t(h)).\]
 The left-hand side of $(*)$ becomes
 \[(t(h)\cdot g'^{-1}, g',p(t(h))(h,g,g')=(h\cdot g'^{-1},g'g, p(t(h))=(h\cdot g'^{-1},g'g, t(g')).\]
 Now
     \[(\psi\circ\phi)(h,g,g')=\psi(h\cdot g'^{-1})=(h\cdot g'^{-1}, p(h\cdot g'^{-1}), p(h\cdot g'^{-1})=(h\cdot g'^{-1},t(g'),t(g')),\]
 \[d(h,g,g')=(h,g,g')^{-1}(h,g,g')=(h^{-1}\cdot g,g^{-1}, g'g)(h,g,g')=\]\[=((h^{-1}\cdot g)(h\cdot g),g^{-1}g,g'g)=(d(h)\cdot g, d(g),g'g)\equiv (d(h)\cdot g,g'g),\]
 and
 \[\theta(d(h)\cdot g,g'g)=(d(h)\cdot g\cdot (g'g)^{-1}, g'g),p(d(h)\cdot g) )=(d(h)\cdot g'^{-1},g'g), d(g)).\]
The right-hand side becomes
 \[(h\cdot g'^{-1},t(g'),t(g'))(d(h)\cdot g'^{-1},g'g,d(g))=(h\cdot g'^{-1}, g'g, t(g')),\]
so $(*)$ is verified.

To show that $(G\times_\rho \Gamma)\rtimes \Gamma$ is similar to $G$, we define
\[\phi: (G\times_\rho \Gamma)\rtimes\Gamma\to G, \; \phi(g,\gamma,\gamma')=g,\]
\[\psi:G\to (G\times_\rho \Gamma)\rtimes\Gamma,\;  \psi(g)=(g, \rho(t(g)),\rho(g) )\]
and 
\[\theta: G^{(0)}\ast \Gamma\to (G\times_\rho \Gamma)\rtimes\Gamma, \theta(u,\gamma)=(u, \rho(u),\gamma^{-1}).\]
We have
\[(\phi\circ \psi)(g)=\phi(g,\rho(t(g)),\rho(g))=g\]
and we need to verify that
\[(**)\;\; \theta(t(g,\gamma,\gamma')) (g,\gamma,\gamma')=(\psi\circ \phi)(g,\gamma,\gamma')\theta(d(g,\gamma,\gamma')).\]
We compute
\[t(g,\gamma, \gamma')=(g,\gamma,\gamma')(g,\gamma,\gamma')^{-1}=(g,\gamma,\gamma')(g^{-1}, \gamma'^{-1}\gamma\rho(g), \gamma'^{-1})=\]
\[=(t(g),\gamma,t(\gamma'))\equiv (t(g),\gamma)\]
and
\[\theta(t(g),\gamma)=(t(g),\rho(t(g),\gamma^{-1}),\]
so the left-hand side of $(**)$  becomes
\[(t(g),\rho(t(g)),\gamma^{-1})(g,\gamma,\gamma')=(g,\rho(t(g)),\gamma^{-1}\gamma').\]
Also
\[(\psi\circ\phi)(g,\gamma,\gamma')=(g,\rho(t(g)),\rho(g)),\]
\[d(g,\gamma,\gamma')=(g,\gamma,\gamma')^{-1}(g,\gamma,\gamma')=(g^{-1}, \gamma'^{-1}\gamma\rho(g), \gamma'^{-1})(g,\gamma,\gamma')=\]
\[=(d(g),\gamma'^{-1}\gamma\rho(g),d(\gamma'))\equiv (d(g), \gamma'^{-1}\gamma\rho(g)),\]
\[\theta(d(g), \gamma'^{-1}\gamma\rho(g))=(d(g),\rho(d(g), \rho(g)^{-1}\gamma^{-1}\gamma'),\]
and the right-hand side is
\[(g,\rho(t(g)),\rho(g))(d(g),\rho(d(g)), \rho(g)^{-1}\gamma^{-1}\gamma')=(g,\rho(t(g)),\gamma^{-1}\gamma'),\]
so $(**)$ is verified.

\end{proof}
For skew products and semi-direct products of groupoids, there are Lyndon--Hochschild--Serre spectral sequences for the  computation of their homology, see \cite{CM} and \cite{M}. 

\begin{thm}\label{seq} Let $G, \Gamma$ be  \'etale groupoids.

(1) Suppose that $\rho : G \to \Gamma$ is a groupoid homomorphism. Then there exists a spectral sequence 
\[E_{p,q}^2 =H_p(\Gamma,H_q(G\times_{\rho}\Gamma))\Rightarrow H_{p+q}(G),\]
where $H_q (G \times_{\rho}\Gamma)$ is regarded as a $\Gamma$-module via the action $\hat{\rho} : \Gamma\curvearrowright G\times_{\rho}\Gamma$.

\bigskip

(2) Suppose that $\varphi:\Gamma \curvearrowright G$ is a groupoid action. Then there exists a spectral sequence 
\[E_{p,q}^2 =H_p(\Gamma,H_q(G))\Rightarrow H_{p+q}(\Gamma\ltimes G),\]
where $H_q(G)$ is regarded as a $\Gamma$-module via the action $\varphi$.
\end{thm}

Note that for $G$ an ample Hausdorff groupoid and $\rho:G\to\ZZ$ a cocycle,  we have the following long exact sequence 
involving the homology of $G$ and the homology of $G\times_\rho\ZZ$, \[0\longleftarrow H_0(G)\longleftarrow H_0(G\times_\rho\ZZ)\stackrel{id-\rho_*}{\longleftarrow}H_0(G\times_\rho\ZZ)\longleftarrow H_1(G)\longleftarrow\cdots\]
\[\cdots\longleftarrow H_n(G)\longleftarrow H_n(G\times_\rho\ZZ)\stackrel{id-\rho_*}{\longleftarrow}H_n(G\times_\rho\ZZ)\longleftarrow H_{n+1}(G)\longleftarrow\cdots\]
Here $\rho_*$ is the map induced by the action $\hat{\rho}: \ZZ\curvearrowright G\times_{\rho}\ZZ$, see Lemma 1.4 in \cite{O}.
\bigskip

\section{Self-similar groupoid actions and their $C^*$-algebra}

\bigskip

We  recall some facts about self-similar groupoid actions and their Cuntz-Pimsner algebras  from \cite{LRRW}.
Let $E = (E^0, E^1, r, s)$ be a finite directed graph  with no sources. For $k \ge 2$, define  the set of paths of length $k$ in $E$ as
\[E^k = \{e_1e_2\cdots e_k : e_i \in E^1,\; r(e_{i+1}) = s(e_i)\}.\]
The maps $r,s$ are naturally extended to $E^k$ by taking
\[r(e_1e_2\cdots e_k)=r(e_1),\;\; s(e_1e_2\cdots e_k)=s(e_k).\]
We denote by $E^* :=\bigcup_{ k\ge 0} E^k$  the space of finite paths (including vertices) and  by $E^\infty$ the infinite path space of $E$ with the usual topology given by the cylinder sets $Z(\alpha)=\{\alpha\xi:\xi\in E^\infty\}$ for $\alpha\in E^*$. 
 
 We can visualize the set  $E^*$ as indexing the vertices of a union of rooted trees  or forest $T_E$   given by $T_E^0 =E^*$ and with edges \[T_E^1 =\{(\mu, \mu e) :\mu\in E^*, e\in E^1\; \text{and}\; s(\mu)=r(e)\}.\]
 
 \begin{example}\label{ex}
 
 For the graph
 
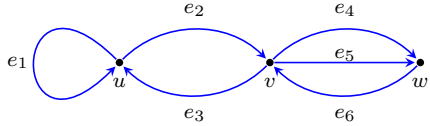
\[\begin{tikzpicture}[shorten >=0.4pt,>=stealth, semithick]
\renewcommand{\ss}{\scriptstyle}
\node[inner sep=1.0pt, circle, fill=black]  (u) at (-2,0) {};
\node[below] at (u.south)  {$\ss u$};
\node[inner sep=1.0pt, circle, fill=black]  (v) at (0,0) {};
\node[below] at (v.south)  {$\ss v$};
\node[inner sep=1.0pt, circle, fill=black]  (w) at (2,0) {};
\node[below] at (w.south)  {$\ss w$};

\draw[->, blue] (u) to [out=45, in=135]  (v);
\node at (-1,0.7){$\ss e_2$};
\draw[->, blue] (v) to [out=-135, in=-45]  (u);
\node at (-1,-0.7) {$\ss e_3$};
\draw[->, blue] (v) to [out=45, in=135]  (w);
\node at (1,0.7){$\ss e_4$};
\draw[->, blue] (v) to (w);
\node at (1,0.1){$\ss e_5$};
\draw[->, blue] (w) to [out=-135, in=-45]  (v);
\node at (1,-0.7) {$\ss e_6$};

\draw[->, blue] (u) .. controls (-3.5,1.5) and (-3.5, -1.5) .. (u);
\node at (-3.35,0) {$\ss e_1$};

\end{tikzpicture}
\]

the forest $T_E$ looks like

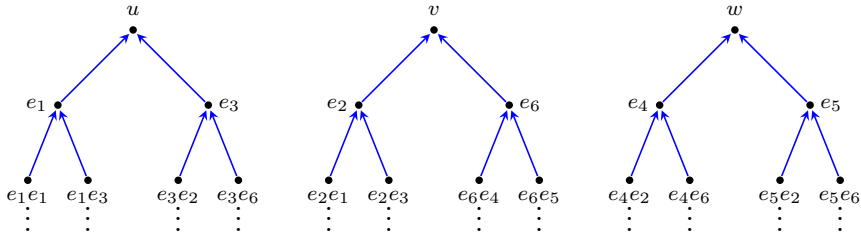
\[
\begin{tikzpicture}[shorten >=0.4pt,>=stealth, semithick]
\renewcommand{\ss}{\scriptstyle}
\node[inner sep=1.0pt, circle, fill=black]  (u) at (-4,4) {};
\node[above] at (u.north)  {$\ss u$};
\node[inner sep=1.0pt, circle, fill=black]  (v) at (0,4) {};
\node[above] at (v.north)  {$\ss v$};
\node[inner sep=1.0pt, circle, fill=black]  (w) at (4,4) {};
\node[above] at (w.north)  {$\ss w$};

\node[inner sep=1.0pt, circle, fill=black] (e1) at (-5,3){};
\node[left] at (-5,3)  {$\ss e_1$};
\node[inner sep=1.0pt, circle, fill=black] (e3) at (-3,3){};
\node[right] at (-3,3)  {$\ss e_3$};
\draw[->, blue] (e1) to  (u);
\draw[-> , blue] (e3) to  (u);
\node[inner sep=1.0pt, circle, fill=black] (e11) at (-5.4,2){};
\node[inner sep=1.0pt, circle, fill=black] (e13) at (-4.6,2){};
\node[below] at (-5.4,2)  {$\ss e_1e_1$};
\node[below] at (-5.4,2) {$\vdots$};
\draw[-> , blue] (e11) to   (e1);
\node[below] at (-4.6,2)  {$\ss e_1e_3$};
\node[below] at (-4.6,2) {$\vdots$};
\draw[-> , blue] (e13) to  (e1);
\node[inner sep=1.0pt, circle, fill=black] (e32) at (-3.4,2){};
\node[inner sep=1.0pt, circle, fill=black] (e36) at (-2.6,2){};
\node[below] at (-3.4,2) {$\ss e_3e_2$};
\node[below] at (-3.4,2) {$\vdots$};
\node[below] at (-2.6,2) {$\ss e_3e_6$};
\node[below] at (-2.6,2) {$\vdots$};
\draw[-> , blue] (e32) to   (e3);
\draw[-> , blue] (e36) to   (e3);

\node[inner sep=1.0pt, circle, fill=black] (e2) at (-1,3){};
\node[left] at (-1,3)  {$\ss e_2$};

\node[inner sep=1.0pt, circle, fill=black] (e6) at (1,3){};
\node[right] at (1,3)  {$\ss e_6$};
\draw[-> , blue] (e2) to  (v);
\draw[-> , blue] (e6) to  (v);
\node[inner sep=1.0pt, circle, fill=black] (e21) at (-1.4,2){};
\node[inner sep=1.0pt, circle, fill=black] (e23) at (-0.6,2){};

\draw[-> , blue] (e21) to   (e2);

\draw[-> , blue] (e23) to  (e2);
\node[inner sep=1.0pt, circle, fill=black] (e64) at (0.6,2){};
\node[inner sep=1.0pt, circle, fill=black] (e65) at (1.4,2){};
\draw[-> , blue] (e64) to   (e6);
\draw[-> , blue] (e65) to   (e6);
\node[below] at (-1.4,2)  {$\ss e_2e_1$};
\node[below] at (-1.4,2) {$\vdots$};
\node[below] at (-0.6,2)  {$\ss e_2e_3$};
\node[below] at (-0.6,2) {$\vdots$};
\node[below] at (0.6,2) {$\ss e_6e_4$};
\node[below] at (0.6,2) {$\vdots$};
\node[below] at (1.4,2) {$\ss e_6e_5$};
\node[below] at (1.4,2) {$\vdots$};

\node[inner sep=1.0pt, circle, fill=black] (e4) at (3,3){};
\node[left] at (3,3)  {$\ss e_4$};

\node[inner sep=1.0pt, circle, fill=black] (e5) at (5,3){};
\node[right] at (5,3)  {$\ss e_5$};
\draw[-> , blue] (e4) to  (w);
\draw[-> , blue] (e5) to  (w);
\node[inner sep=1.0pt, circle, fill=black] (e42) at (2.6,2){};
\node[inner sep=1.0pt, circle, fill=black] (e46) at (3.4,2){};

\draw[-> , blue] (e42) to   (e4);

\draw[-> , blue] (e46) to  (e4);
\node[inner sep=1.0pt, circle, fill=black] (e52) at (4.6,2){};
\node[inner sep=1.0pt, circle, fill=black] (e56) at (5.4,2){};
\draw[-> , blue] (e52) to   (e5);
\draw[-> , blue] (e56) to   (e5);
\node[below] at (2.6,2)  {$\ss e_4e_2$};
\node[below] at (2.6,2) {$\vdots$};
\node[below] at (3.4,2)  {$\ss e_4e_6$};
\node[below] at (3.4,2) {$\vdots$};
\node[below] at (4.6,2) {$\ss e_5e_2$};
\node[below] at (4.6,2) {$\vdots$};
\node[below] at (5.4,2) {$\ss e_5e_6$};
\node[below] at (5.4,2) {$\vdots$};

\end{tikzpicture}
\]

 \end{example}

Recall that a partial isomorphism of the forest $T_E$ corresponding to a given directed graph $E$  consists of a pair $(v, w) \in E^0 \times E^0$ and a bijection $g : vE^* \to wE^*$ such that
\begin{itemize}

\item $g|_{vE^k} : vE^k \to wE^k$ is bijective for all $k\ge 1$.

\item $g(\mu e)\in g(\mu)E^1$ for $\mu\in vE^*$  and $e\in E^1$ with $r(e)=s(\mu)$.
 
 \end{itemize}
 
The set of partial isomorphisms of $T_E$ forms a groupoid PIso$(T_E)$ with unit space $E^0$. The identity morphisms are  $id_v : vE^* \to vE^*$, the inverse of $g : vE^* \to wE^*$ is $g^{-1} : wE^* \to vE^*$, and the  multiplication is composition. We often identify $v\in E^0$ with $id_v\in$ PIso$(T_E)$.

\begin{dfn}\label{ss} 
Let $E$ be a finite directed graph with no sources, and let $G$ be a groupoid with unit space $E^0$. A {\em self-similar action} $(G,E)$ on the path space of $E$  is given by a faithful groupoid homomorphism $G\to$ PIso$(T_E)$ such that for every $g\in G$ and every $e\in d(g)E^1$ there exists a unique $h\in G$ denoted  by $g|_e$ and called the restriction of $g$ to $e$ such that
\[g\cdot(e\mu)=(g\cdot e)(h\cdot \mu)\;\;\text{for all}\;\; \mu\in s(e)E^*.\]
\end{dfn}
\begin{rmk}
It is possible that $g|_e=g$ for all $e\in d(g)E^1$, in which case \[g\cdot(e_1e_2\cdots e_n)=(g\cdot e_1)\cdots(g\cdot e_n).\]
We have \[d(g|_e)=s(e),\; t(g|_e)=s(g\cdot e)=g|_e\cdot s(e),\;  r(g\cdot e)=g\cdot r(e).\] In particular, the source map may not be equivariant as  in \cite{EP}. It is shown in Appendix A of \cite{LRRW} that a self-similar group action $(G,E)$ as in \cite{EP} determines a self-similar groupoid action $(E^0\rtimes G, E)$ as in Definition \ref{ss}, where $E^0\rtimes G$ is the semi-direct product or the action groupoid of the group $G$ acting on $E^0$. Note that not any self-similar groupoid action  comes from a self-similar group action, as seen in our example below.
\end{rmk}

\begin{prop} A self-similar groupoid action $(G,E)$ as above extends to an action of $G$ on the path space $E^*$ and  determines an action of $G$ on the graph $T_E$, in the sense that $G$ acts on both the vertex space $T_E^0$ and the edge space $T_E^1$ and intertwines the range and the source maps of $T_E$, see Definition 4.1 in \cite{De}.
\end{prop}
\begin{proof}
Indeed, the vertex space $T_E^0=E^*$ is fibered over $G^{(0)}=E^0$ via the map $\mu\mapsto r(\mu)$. For $(\mu, \mu e)\in T^1_E$ we set $s(\mu, \mu e)=\mu e$ and $r(\mu, \mu e)=\mu$. Since   $r(\mu e)=r(\mu)$, the edge space $T^1_E$ is also fibered over $G^{(0)}$. The action of $G$ on $T^1_E$ is given by
\[g\cdot (\mu, \mu e)=(g\cdot \mu, g\cdot (\mu e))\;\;\text{when}\;\; d(g)=r(\mu).\]
Since
\[s(g\cdot(\mu, \mu e))=s(g\cdot \mu, g\cdot (\mu e))=g\cdot (\mu e)=g\cdot s(\mu, \mu e)\]
and
\[r(g\cdot(\mu, \mu e))=r(g\cdot \mu, g\cdot (\mu e))=g\cdot \mu =g\cdot r(\mu, \mu e),\]
the  actions on $T_E^0$ and $T_E^1$ are compatible.

\end{proof}

The faithfulness condition ensures that for each $g \in G$ and each $\mu\in E^*$ with $d(g) = r(\mu)$, there is a unique element $g|_\mu\in G$ satisfying 
\[g\cdot(\mu\nu)=(g\cdot\mu)(g|_\mu\cdot \nu)\;\text{ for all}\; \nu\in s(\mu)E^*.\] 
By Proposition 3.6 of \cite{LRRW}, self-similar groupoid actions have the following properties: for $g, h \in G, \mu\in d(g)E^*$, and $\nu\in s(\mu)E^*$,

(1) $g|_{\mu\nu} = (g|_\mu)|_\nu$;

(2) $id_{r(\mu)}|_{\mu} = id_{s(\mu)}$; 

(3) if $(h, g)\in G^{(2)}$, then $(h|_{g\cdot\mu},g|_\mu)\in G^{(2)}$ 
and $(hg)|_\mu = (h|_{g\cdot\mu})(g|_\mu)$; 

(4) $g^{-1}|_\mu=(g|_{g^{-1}\cdot\mu})^{-1}$.

\begin{dfn}
The $C^*$-algebra $C^*(G,E)$  of a self-similar action $(G,E)$ is defined as the Cuntz-Pimsner algebra of the $C^*$-correspondence \[\Mm=\Mm(G,E)=\Xx(E)\otimes_{C(E^0)}C^*(G)\] over $C^*(G)$.   Here $\Xx(E)=C(E^1)$ is the $C^*$-correspondence over $C(E^0)$ associated to the graph $E$ and $C(E^0)=C(G^{(0)})\subseteq C^*(G)$. The right action of $C^*(G)$ on $\Mm$ is the usual one and the left action is determined by the  representation
 \[W:G\to \Ll(\Mm), \;\; W_g(i_e\otimes a)=\begin{cases} i_{g\cdot e}\otimes i_{g|_e}a\;\;\text{if}\; d(g)=r(e)\\0\;\;\text{otherwise,}\end{cases}\]
 where $g\in G, i_e\in C(E^1)$ and $i_g\in C_c(G)$ are point masses and  $a\in C^*(G)$. The inner product of $\Mm$ is given by
\[\la \xi\otimes a,\eta\otimes b\ra=\la\la\eta,\xi\ra a,b\ra=a^*\la \xi,\eta\ra b\]
for $\xi, \eta\in C(E^1)$ and $a,b\in C^*(G)$. 
\end{dfn}
\begin{rmk}
Recall that the operations on $\Xx(E)$ are given by
\[(\xi\cdot a)(e)=\xi(e)a(s(e)), \;\la \xi, \eta\ra(v)=\sum_{s(e)=v}\overline{\xi(e)}\eta(e),\; (a\cdot \xi)(e)=a(r(e))\xi(e)\]
for $a\in C(E^0)$ and $ \xi, \eta\in C(E^1)$.  The elements $i_e\otimes 1$ for $e\in E^1$ form a Parseval frame for $\Mm$ and every $\zeta\in \Mm$ is a finite sum \[\zeta=\sum_{e\in E^1}i_e\otimes\la i_e\otimes 1,\zeta\ra.\]
In particular, if $\Xx(E)^*$ denotes the dual $C^*$-correspondence, then 
\[\Ll(\Mm)=\Kk(\Mm)\cong \Xx(E)\otimes_{C(E^0)}C^*(G)\otimes_{C(E^0)}\Xx(E)^*\cong M_n\otimes C^*(G),\] where $n=|E^1|$. The isomorphism is given by
\[i_{e_j}\otimes i_g\otimes i_{e_k}^*\mapsto e_{jk}\otimes i_g\] for $E^1=\{e_1,...,e_n\}$ and for matrix units $e_{jk}\in M_n$. There is a unital homomorphism $\Kk(\Xx(E))\to \Kk(\Mm)$ given by \[i_e\otimes i_f^*\to i_e\otimes 1\otimes i_f^*.\]

 Since our groupoids have finite unit space $E^0$, the orbit space for the canonical action of $G$ on $E^0$ is finite, and $C^*(G)$ is the direct sum of $C^*$-algebras of transitive groupoids. Each such transitive groupoid will be isomorphic to a groupoid of the form $V\times H\times V$ with the usual operations, for some subset $V\subseteq E^0$ and isotropy group $H$, hence its $C^*$-algebra will be  isomorphic to $C^*(H)\otimes M_{|V|}$.

\end{rmk}
We recall the following result, see Propositions 4.4 and 4.7 in \cite{LRRW}.
\begin{thm}\label{gen}
If $U_g, P_v$ and $T_e$ are the images of $g\in G, v\in E^0=G^{(0)}$ and of $e\in E^1$ in the Cuntz-Pimsner algebra $C^*(G,E)$, then
\begin{itemize}

\item $g\mapsto U_g$ is a representation by partial isometries of $G$ with $U_v=P_v$ for $v\in E^0$;

\item $T_e$ are partial isometries with $T_e^*T_e=P_{s(e)}$ and $\ds \sum_{r(e)=v}T_eT_e^*=P_v$;

\item  $U_gT_e=\begin{cases}T_{g\cdot e}U_{g|_e}\;\mbox{if}\; d(g)=r(e)\\0,\;\mbox{otherwise}\end{cases}$ and $\;\; U_gP_v=\begin{cases}P_{g\cdot v}U_g\;\mbox{if}\; d(g)=v\\0,\;\mbox{otherwise.}\;\end{cases}$
\end{itemize}
There is a gauge action $\gamma$ of $\TT$ on $C^*(G,E)$ such that $\gamma_z(U_g)=U_g,$ and $\gamma_z(T_e)=zT_e$ for $z\in \TT$.

Given $\mu=e_1\cdots e_n\in E^*$ with $e_i\in E^1$, we let $T_\mu:=T_{e_1}\cdots T_{e_n}$. Then $C^*(G,E)$ is the closed linear span of elements $T_\mu U_gT_\nu^*$, where $\mu, \nu\in E^*$ and $g\in G_{s(\nu)}^{s(\mu)}$.
\end{thm}

For each $k\ge 1$, consider $\Ff_k$ the closed linear span of elements $T_\mu U_gT_\nu^*$ with $\mu,\nu\in E^k$  and $g\in G_{s(\nu)}^{s(\mu)}$. Then the fixed point algebra  $\Ff(G,E):=C^*(G,E)^{\TT}$ under the gauge action is isomorphic to $\ds \varinjlim \Ff_k$. We have \[\Ff_k\cong \Ll(\Mm^{\otimes k})\cong \Xx(E)^{\otimes k}\otimes_{C(E^0)}C^*(G)\otimes_{C(E^0)}\Xx(E)^{*\otimes k}\]
using the map $T_\mu U_gT_\nu^*\mapsto i_\mu\otimes i_g\otimes i_\nu^*$, where  $i_\mu\in \Xx(E)^{\otimes k}=C(E^k)$ are point masses. The embeddings $\Ff_k\hookrightarrow \Ff_{k+1}$ are determined by the map \[\phi=\phi_W:C^*(G)\to \Ll(\Mm),\;\;  \phi_W(i_g)=W_g.\] In particular, for $a\in C^*(G)$ we get
\[\phi(a)=\sum_{e\in E^1}\theta_{i_e\otimes 1, a^*(i_e\otimes 1)},\]
where $\theta_{\xi, \eta}(\zeta)=\xi\la \eta, \zeta\ra$. 
The embeddings
$\Ff_k\hookrightarrow \Ff_{k+1}$ are then
\[\phi_k(i_{\mu}\otimes i_g\otimes i_\nu^*)=\begin{cases}\ds \sum_{x\in d(g)E^1}i_{\mu y} \otimes i_{g|_x}\otimes i_{\nu x}^*,\;\text{if}\; g\in G_{s(\nu)}^{s(\mu)}\;\text{and}\; g\cdot x=y \\0,\;\text{ otherwise.}\end{cases}\]
\begin{rmk}
The $C^*$-algebra $C^*(G,E)$ can be described as the crossed product of $\Ff(G,E)$ by an endomorphism and in many cases, knowledge about $K_*(\Ff_k)$ is sufficient to determine $K_*(\Ff(G,E))$ and $K_*(C^*(G,E))$.
For the case when $G$ is a group, see section 3 in \cite{N1}.
In the particular case when $g|_e=g$ for all $g\in G$ and $e\in d(g)E^1$ we have $C^*(G,E)\cong C^*(E)\rtimes G$.
\end{rmk}
\bigskip

%%%%%5
\section{Exel-Pardo groupoids for self-similar actions}

\bigskip

In this section, we generalize results from \cite{EP} and we define the groupoid associated to a self-similar action of a groupoid $G$ on the path space of a finite directed graph $E$ with no sources.

 As in \cite{EP},  we first define the   inverse semigroup 
\[\Ss(G,E)=\{(\alpha, g, \beta): \alpha, \beta\in E^*, g\in G_{s(\beta)}^{s(\alpha)} \}\cup\{0\}\]
associated to the self-similar  action $(G,E)$, with operations
\[(\alpha, g, \beta)(\lambda, h, \omega)=\begin{cases}(\alpha, g(h|_{h^{-1}\cdot\mu}), \omega(h^{-1}\cdot \mu))&\;\text{if}\; \beta=\lambda\mu\\(\alpha(g\cdot\mu), g|_{\mu}h, \omega)&\;\text{if}\; \lambda=\beta\mu\\0&\;\text{otherwise}\end{cases}\]
and $(\alpha, g, \beta)^*=(\beta, g^{-1}, \alpha)$ for $\alpha, \beta, \lambda, \omega\in E^*$. These operations make sense since

\[ d(g)=s(\beta)=t(h|_{h^{-1}\cdot\mu})\;\text{and}\; d(g(h|_{h^{-1}\cdot\mu}))=s(\omega(h^{-1}\cdot \mu)) \;\text{when}\; \beta=\lambda\mu,\]\[d(g|_\mu)=s(\mu)=t(h)\;\text{and}\; d(g|_{\mu}h)=s(\omega)\; \;\text{when}\;\lambda=\beta\mu .\]
Note that $(\alpha,g,\beta)(\beta, h, \omega)=(\alpha, gh,\omega)$ and the nonzero idempotents are of the form $z_\alpha=(\alpha, s(\alpha),\alpha)$.

The inverse semigroup $\Ss(G,E)$ acts on the infinite path space $E^\infty$ by partial homeomorphisms. The action of $(\alpha, g, \beta)\in \Ss(G,E)$ on $\xi=\beta\mu\in \beta E^\infty$ is given by \[(\alpha, g, \beta)\cdot \beta\mu=\alpha(g\cdot\mu)\in \alpha E^\infty.\] The action of $G$ on $E^\infty$ is defined by $g\cdot \mu=\eta$, where for all $n$ we have $\eta_1\cdots \eta_n=g\cdot(\mu_1\cdots\mu_n)$. Note that $r(g\cdot \mu)=g\cdot r(\mu)=g\cdot s(\beta)=s(\alpha)$, so the action is well defined.

The groupoid of germs associated with $(\Ss(G,E), E^\infty)$ is
\[\Gg(G,E)=\{[\alpha, g, \beta; \xi]: \alpha, \beta\in E^*,\; g\in G^{s(\alpha)}_{s(\beta)},\; \xi\in \beta E^\infty\}.\]
Two germs $[\alpha, g,\beta;\xi], [\alpha',g',\beta';\xi']$ in $\Gg(G,E)$ are equal if and only if $\xi=\xi'$ and there exists a neighborhood $V$ of $\xi$ such that $(\alpha, g, \beta)\cdot \eta=(\alpha', g', \beta')\cdot \eta$ for all $\eta\in V$. 
We obtain that $\xi=\beta\lambda\zeta$ for $\lambda\in E^*$ and $\zeta\in E^\infty$, with $r(\lambda)=s(\beta)$ and $r(\zeta)=s(\lambda)$. Moreover,  \[\alpha'=\alpha(g\cdot \lambda), \beta'=\beta\lambda,\;\mbox{and}\; g'=g|_\lambda.\]
The unit space of $\Gg(G,E)$ is
\[\Gg(G,E)^{(0)}=\{[\alpha, s(\alpha), \alpha; \xi]: \xi\in \alpha E^\infty\},\] identified with $E^\infty$ by the map $[\alpha, s(\alpha), \alpha; \xi]\mapsto \xi$. 

The terminus and domain maps of the groupoid $\Gg(G,E)$ are given by
\[t([\alpha, g,\beta; \beta\mu])=\alpha(g\cdot\mu),\;\; d([\alpha, g, \beta;\beta\mu])=\beta\mu.\]
If two elements $\gamma_1,\gamma_2\in \Gg(G,E)$ are composable, then \[\gamma_1=[\alpha_1, g_1, \alpha_2; \alpha_2(g_2\cdot\xi)],\;\; \gamma_2=[\alpha_2, g_2, \beta; \beta\xi]\] for some $\alpha_1, \alpha_2, \beta\in E^*, \xi\in E^\infty, (g_1, g_2)\in G^{(2)}$ and in this case
\[\gamma_1\gamma_2=[\alpha_1, g_1g_2, \beta;\beta\xi].\]
In particular, 
\[[\alpha, g,\beta;\beta\mu]^{-1}=[\beta, g^{-1}, \alpha;\alpha(g\cdot \mu)].\]

The topology on $\Gg(G,E)$ is generated by the compact open bisections of the form
\[B(\alpha, g, \beta; V)=\{[\alpha, g, \beta; \xi]\in \Gg(G,E): \xi=\beta\zeta\in V\},\]
where $\alpha, \beta\in E^*, g\in G_{s(\beta)}^{s(\alpha)}$ are fixed, and $V\subseteq Z(\beta)=\beta E^\infty$ is an open subset. 

\begin{dfn}
A self-similar groupoid action $(G,E)$ is called pseudo free if for every $g\in G$  and every $e\in d(g)E^1$, the condition $g\cdot e=e$ and $g|_{e}=s(e)$ implies that $g=r(e)$.
 \end{dfn}
\begin{rmk}
If $(G,E)$ is pseudo free,  then $g_1\cdot \alpha=g_2\cdot \alpha$ and $g_1|_\alpha=g_2|_\alpha$ for some $\alpha\in E^*$ implies $g_1=g_2$.
\end{rmk}
\begin{proof}
Indeed, since $g_2^{-1}g_1\cdot\alpha=\alpha$ and $g_2^{-1}g_1|_\alpha=g_2^{-1}|_{g_1\cdot\alpha}g_1|_\alpha=d(g_1|_\alpha)=s(\alpha)$, it follows that $g_2^{-1}g_1=r(\alpha)$, so $g_1=g_2$.
\end{proof}

\begin{thm}
If the action of $G$ on $E$ is pseudo free, then the groupoid $\Gg(G,E)$ is Hausdorff and  its $C^*$-algebra $C^*(\Gg(G,E))$
is isomorphic to the Cuntz-Pimsner algebra $C^*(G, E)$.
\end{thm}
 \begin{proof}
 Since $(G,E)$ is pseudo free, it follows that $[\alpha,g,\beta;\xi]=[\alpha, g',\beta;\xi]$ if and only if $g=g'$. Moreover, the groupoid $\Gg(G,E)$ is Hausdorff, see Proposition 12.1 in \cite{EP}. Using the properties given in Theorem \ref{gen} and the groupoid multiplication, the isomorphism $\phi: C^*(G,E)\to C^*(\Gg(G,E))$ is given by
 \[\phi(P_v)=\chi_{B(v, v,v;Z(v))},\]\[ \phi(T_e)=\chi_{B(e, s(e),s(e);Z(s(e)))},\]\[\phi(U_g)=\chi_{B(t(g),g,d(g);Z(d(g)))}\]
 for $v\in E^0, e\in E^1$ and $g\in G$. Here $\chi_A$ is the indicator function of $A$.
 \end{proof}

Recall that ample Hausdorff groupoids which are similar or Morita equivalent have isomorphic homology, see Lemma 4.3 and Theorem 3.12 in \cite{FKPS}. The general strategy of computing the homology of the ample  groupoid $\Gg(G,E)$ is the following. 

There is a cocycle $\rho: \Gg(G,E)\to \ZZ$ given by $[\alpha, g, \beta; \xi]\mapsto |\alpha|-|\beta|$ with kernel
\[\Hh(G,E)=\{[\alpha, g, \beta; \xi]\in \Gg(G,E) :|\alpha|=|\beta|\}.\]

It follows from Theorem \ref{seq} that  we have a spectral sequence 
\[E_{p,q}^2=H_p(\ZZ,H_q(\Hh(G,E)))\Rightarrow H_{p+q}(\Gg(G,E)).\]
Now $\Hh(G,E)=\bigcup_{k\ge 1} \Hh_k(G,E)$ where
\[\Hh_k(G,E)=\{[\alpha, g, \beta; \xi]\in \Gg(G,E) :|\alpha|=|\beta|=k\}.\]
There are groupoid homomorphisms \[\tau_k:\Hh_k(G,E)\to G,\; \tau_k([\alpha,g,\beta;\xi])=g\] and $\ker\tau_k$ is AF for all $k\ge 1$. Indeed, consider $R_k$ the equivalence relation on $E^k$ such that $(\alpha, \beta)\in R_k$ if there is $g\in G$ with $g\cdot s(\beta)=s(\alpha)$. Then the map $[\alpha, g, \beta;\xi]\mapsto ((\xi, g),(\alpha, \beta))$ gives an isomorphism between $\Hh_k(G,E)$ and $(E^\infty\rtimes G)\times R_k$, so $\ker\tau_k$ is isomorphic to $E^\infty\times R_k$. 
It follows that we have another spectral sequence 
\[E_{p,q}^2=H_p(G,H_q(\ker\tau_k))\Rightarrow H_{p+q}(\Hh_k(G,E)).\]
It is known that $H_0(\ker\tau_k)\cong K_0(C^*(\ker \tau_k))$ and $H_q(\ker\tau_k)=0$ for $k\ge 1$. Also, $\ker\tau_k$ is similar with $\Hh_k(G,E)\times_{\tau_k}G$.
Assuming that we computed $H_q(\Hh_k(G,E))$ for all $k$, then
\[H_q(\Hh(G,E))=\varinjlim_{k\to \infty}H_q(\Hh_k(G,E))\]
can be computed using the inclusion maps \[j_k:\Hh_k(G,E)\hookrightarrow \Hh_{k+1}(G,E),\;\; j_k([\alpha, g, \beta; \beta x\mu])=[\alpha y, g|_x,\beta x; \beta x\mu],\] where $x\in E^1$  and  $g\cdot x=y$.  
\bigskip

%%%%
\section{example}

\bigskip

Consider again the graph from Example \ref{ex}

\[\begin{tikzpicture}[shorten >=0.4pt,>=stealth, semithick]
\renewcommand{\ss}{\scriptstyle}
\node[inner sep=1.0pt, circle, fill=black]  (u) at (-2,0) {};
\node[below] at (u.south)  {$\ss u$};
\node[inner sep=1.0pt, circle, fill=black]  (v) at (0,0) {};
\node[below] at (v.south)  {$\ss v$};
\node[inner sep=1.0pt, circle, fill=black]  (w) at (2,0) {};
\node[below] at (w.south)  {$\ss w$};

\draw[->, blue] (u) to [out=45, in=135]  (v);
\node at (-1,0.7){$\ss e_2$};
\draw[->, blue] (v) to [out=-135, in=-45]  (u);
\node at (-1,-0.7) {$\ss e_3$};
\draw[->, blue] (v) to [out=45, in=135]  (w);
\node at (1,0.7){$\ss e_4$};
\draw[->, blue] (v) to (w);
\node at (1,0.1){$\ss e_5$};
\draw[->, blue] (w) to [out=-135, in=-45]  (v);
\node at (1,-0.7) {$\ss e_6$};

\draw[->, blue] (u) .. controls (-3.5,1.5) and (-3.5, -1.5) .. (u);
\node at (-3.35,0) {$\ss e_1$};

\end{tikzpicture}
\]
with $E^0=\{u,v,w\}$ and $ E^1=\{e_1, e_2, e_3, e_4, e_5, e_6\}$. 

Consider the groupoid $G$ with unit space $G^{(0)}=\{u,v,w\}$ and generators $a,b,c$  where $d(a)=u, t(a)=d(b)=v, d(c)=t(b)=w$.

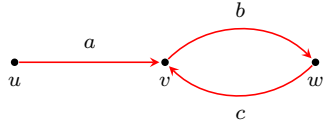
\[\begin{tikzpicture}[shorten >=0.4pt,>=stealth, semithick]
\renewcommand{\ss}{\scriptstyle}
\node[inner sep=1.0pt, circle, fill=black]  (u) at (-2,0) {};
\node[below] at (u.south)  {$\ss u$};
\node[inner sep=1.0pt, circle, fill=black]  (v) at (0,0) {};
\node[below] at (v.south)  {$\ss v$};
\node[inner sep=1.0pt, circle, fill=black]  (w) at (2,0) {};
\node[below] at (w.south)  {$\ss w$};
\draw[->, red] (u) to (v);
 \node at (-1,0.25) {$\ss a$}; 

\draw[->, red] (v) to [out=45, in=135]  (w);
\node at (1,0.7){$\ss b$};
\draw[->, red] (w) to [out=-135, in=-45]  (v);
\node at (1,-0.7) {$\ss c$};

\end{tikzpicture}
\]
We define the action of $G$ by
 \[a\cdot e_1=e_2,\;\; a|_{e_1}=u,\;\; a\cdot e_3=e_6,\;\; a|_{e_3}=b,\]
\[b\cdot e_2=e_5,\;\; b|_{e_2}=a, \;\; b\cdot e_6=e_4, \;\; b|_{e_6}=c,\] 
\[c\cdot e_4=e_2,\;\; c|_{e_4}=a^{-1},\;\; c\cdot e_5=e_6,\;\; c|_{e_5}=b.\]
The actions of $a^{-1}, b^{-1}, c^{-1}$ and their restrictions are then uniquely determined:
\[a^{-1}\cdot e_2=e_1,\;\; a^{-1}|_{e_2}=u,\;\; a^{-1}\cdot e_6=e_3,\;\; a^{-1}|_{e_6}=b^{-1},\]
\[b^{-1}\cdot e_5=e_2,\;\;b^{-1}|_{e_5}=a^{-1},\;\; b^{-1}\cdot e_4=e_6,\;\; b^{-1}|_{e_4}=c^{-1},\]
\[c^{-1}\cdot e_2=e_4,\;\; c^{-1}|_{e_2}=a,\;\; c^{-1}\cdot e_6=e_5,\;\;c^{-1}|_{e_6}=b^{-1}.\]  
The actions of the units $u,v,w$ and their restrictions are 
\[u\cdot e_1=e_1,\; u|_{e_1}=u, \; u\cdot e_3=e_3, \; u|_{e_3}=v, \;  v\cdot e_2=e_2, \; v|_{e_2}=u,\]
\[ v\cdot e_6=e_6, \; v|_{e_6}=w,\;  w\cdot e_4=e_4, \; w|_{e_4}=v,\;  w\cdot e_5=e_5, \; w|_{e_5}=v.\]

This data determine a pseudo free self-similar action of $G$ on the path space of $E$. We can  characterize the action  by the formulas
\[a\cdot e_1\mu=e_2\mu, \;  \; a\cdot e_3\mu =e_6(b\cdot\mu),\; \; b\cdot(e_2\mu)=e_5(a\cdot\mu), \]\[ b\cdot e_6\mu=e_4(c\cdot\mu),\;\; c\cdot e_4\mu=e_2(a^{-1}\cdot\mu), \;\; c\cdot e_5\mu=e_6(b\cdot\mu)\]
where $\mu\in E^*$, and these  determine uniquely an action of  $G$ on $E^*$ and on the graph $T_E$.

We will prove  that $G$ is a transitive groupoid with isotropy isomorphic to $\ZZ$, hence $C^*(G)\cong M_3(C(\TT))$ since $|E^0|=3$. Indeed, there is only one orbit for the action of $G$ on its unit space, and let's show that the cyclic group $G_u^u=\la a^{-1}cba\ra$ is isomorphic to $\ZZ$. Since
\[(a^{-1}cba)\cdot e_1=(a^{-1}cb)\cdot e_2=(a^{-1}c)\cdot e_5=a^{-1}\cdot e_6=e_3\]
and
\[(a^{-1}cba)\cdot e_3=(a^{-1}cb)\cdot e_6=(a^{-1}c)\cdot e_4=a^{-1}\cdot e_2=e_1,\]
it follows that $(a^{-1}cba)^n$ is not the identity for $n$ odd. Now
\[(a^{-1}cba)|_{e_1}=(a^{-1}cb)|_{a\cdot e_1}a|_{e_1}=(a^{-1}cb)|_{e_2}(a^{-1}c)|_{b\cdot e_2}b|_{e_2}=\]
\[=(a^{-1}c)|_{e_5}a=a^{-1}|_{c\cdot e_5}c|_{e_5}a=(a|_{a^{-1}\cdot e_6})^{-1}ba=b^{-1}ba=a\]
and similarly 
\[(a^{-1}cba)|_{e_3}=a^{-1}cb.\]
We deduce
\[(a^{-1}cba)^2|_{e_1}=(a^{-1}cba)|_{(a^{-1}cba)\cdot e_1}(a^{-1}cba)|_{e_1}=(a^{-1}cba)|_{e_3}a=a^{-1}cba.\]
By induction,
\[(a^{-1}cba)^{2k}|_{e_1}=(a^{-1}cba)^k.\]
We consider the action of $(a^{-1}cba)^{2k}$ on sufficiently long paths of the form $\mu=e_1\cdots e_1$ and after repeatedly reducing  by factors of $2$, we arrive at $(a^{-1}cba)^{2k}|\mu=(a^{-1}cba)^m$ with $m$ odd, in particular
\[(a^{-1}cba)^{2k}\cdot \mu e_1=\mu(a^{-1}cba)^m\cdot e_1=\mu e_3,\]
so $(a^{-1}cba)^n$ is not the identity for $n$ even.
It follows that $G_u^u=\la a^{-1}cba\ra$ is isomorphic to $\ZZ$.

An isomorphism of $G$ with the groupoid $G^{(0)}\times \ZZ\times G^{(0)}$ is given by the map
\[a\mapsto (v,1,u),\; b\mapsto (w,1,v),\; c\mapsto (v,1,w),\]
and $G_u^u\cong \{(u,k,u): k\in 2\ZZ\}\cong \ZZ$.
In this case, since $|E^1|=6$ and $C^*(G)\cong M_3(C(\TT))$, it follows that \[\Ff_k\cong \Ll(\Mm^{\otimes k})\cong \Xx(E)^{\otimes k}\otimes_{C(E^0)} C^*(G)\otimes_{C(E^0)}\Xx(E)^{*\otimes k}\cong  M_{3\cdot 6^k}(C(\TT)),\] so $K_0(\Ff_k)\cong \ZZ\cong K_1(\Ff_k)$ and $\Ff(G,E)$ is an $A\TT$-algebra. Recall that the embeddings
$\Ff_k\hookrightarrow \Ff_{k+1}$ are determined by
\[\phi_k(i_{\mu}\otimes i_g\otimes i_\nu^*)=\begin{cases}\ds \sum_{x\in d(g)E^1}i_{\mu y} \otimes i_{g|_x}\otimes i_{\nu x}^*,\;\text{if}\; g\in G_{s(\nu)}^{s(\mu)}\;\text{and}\; g\cdot x=y \\0,\;\text{ otherwise.}\end{cases}\]
 In particular,
\[i_{\mu}\otimes i_a\otimes i_{\nu}^*\mapsto i_{\mu e_2}\otimes i_u\otimes i_{\nu e_1}^*+i_{\mu e_6}\otimes i_b\otimes i_{\nu e_3}^*,\]
\[i_{\mu}\otimes i_b\otimes i_{\nu}^*\mapsto i_{\mu e_5}\otimes i_a\otimes i_{\nu e_2}^*+i_{\mu e_4}\otimes i_c\otimes i_{\nu e_6}^*,\]
\[i_{\mu}\otimes i_c\otimes i_{\nu}^*\mapsto i_{\mu e_2}\otimes i_{a^{-1}}\otimes i_{\nu e_4}^*+i_{\mu e_6} \otimes i_b\otimes i_{\nu e_5}^*.\]

To compute the $K$-theory of $\Ff(G,E)$, we first determine the maps \[\Phi_i=[\phi_k]_i:K_i(C^*(G))\cong \ZZ\to K_i(\Ll(\Mm))\cong \ZZ\] for $i=0,1$. Since $K_0(C^*(G))$ is generated by $[i_u]$ and 
\[i_u\mapsto i_{e_1}\otimes i_u\otimes i_{e_1}^*+i_{e_3}\otimes i_v\otimes i_{e_3}^*,\]
it follows that $\Phi_0$ is multiplication  by $2$. Since $K_1(C^*(G))$ is generated by $[zi_u+i_v+i_w]$ and
\[zi_u+i_v+i_w\mapsto z(i_{e_1}\otimes i_u\otimes i_{e_1}^*+i_{e_3}\otimes i_v\otimes i_{e_3}^*)+\]\[+i_{e_2}\otimes i_u\otimes i_{e_2}^*+i_{e_6}\otimes i_w\otimes i_{e_6}^*+i_{e_4}\otimes i_v\otimes i_{e_4}^*+i_{e_5}\otimes i_v\otimes i_{e_5}^*,\]
it follows that $\Phi_1$ is also multiplication by $2$.
We obtain
\[K_0(\Ff(G,E))\cong \ZZ[1/2]\cong K_1(\Ff(G,E)).\]
We can describe $C^*(G,E)$ as the crossed product of the simple $A\TT$-algebra $\Ff(G,E)$ by an endomorphism. Using \cite{De1}, its $K$-theory is given by
\[K_0(C^*(G,E))\cong \ker(id-\Phi_1)\oplus \ZZ/(id-\Phi_0)\ZZ\cong 0,\]\[K_1(C^*(G,E))\cong \ker(id-\Phi_0)\oplus \ZZ/(id-\Phi_1)\ZZ\cong 0.\]

\medskip

Now we compute the homology of the Exel-Pardo groupoid $\Gg(G,E)$.
Its unit space $E^\infty$  is a disjoint union of three Cantor sets $uE^\infty\cup vE^\infty\cup wE^\infty$. The kernel of the cocycle $\rho:\Gg(G,E)\to \ZZ,\;\; \rho([\alpha, g,\beta;\xi])=|\alpha|-|\beta|$ is the minimal groupoid 
\[\Hh(G,E)=\bigcup_{k\ge 1} \Hh_k(G,E),\]
where \[\Hh_k(G,E)=\{[\alpha, g,\beta;\xi]\in \Gg(G,E): |\alpha|=|\beta|=k\}\] is isomorphic to $(E^\infty\rtimes G)\times R_k$ via the map $[\alpha, g, \beta; \xi]\mapsto ((\xi, g), (\alpha, \beta))$, where $R_k$ is the equivalence relation on $E^k$ given by $(\alpha, \beta)\in R_k$ if there is $g\in G$ with $s(\alpha)=g\cdot s(\beta)$. Since $G$ is transitive, $C^*(R_k)\cong M_{6^k}$ and it follows that the groupoid $\Hh_k(G,E)$ is equivalent with $E^\infty\rtimes G$. There is a groupoid homomorphism $\tau_k:\Hh_k(G,E)\to G$ given by $ [\alpha, g,\beta;\xi]\mapsto g$ with kernel equivalent with the space $E^\infty$. Since the groupoid $G$ is Morita equivalent with the group $\ZZ$, we deduce
\[H_q(\Hh_k(G,E))\cong H_q(E^\infty\rtimes G)\cong H_q(G,C(E^\infty, \ZZ))\cong H_q(\ZZ, C(uE^\infty,\ZZ)).\]
It follows that
\[H_0(\Hh_k(G,E))\cong \ker(id-\sigma_*),\;\; H_1(\Hh_k(G,E))\cong \text{coker}(id-\sigma_*),\]
where $\sigma_*$ is induced by the action of $G$ on $C(E^\infty, \ZZ)$
and $H_q(\Hh_k(G,E)\cong 0$ for $q\ge 2$. Since the action of $G$ on $E^\infty$ is free and transitive, it follows that
\[H_0(\Hh_k(G,E))\cong H_1(\Hh_k(G,E))\cong C(uE^\infty,\ZZ).\]
Note that $H_0(\Hh_k(G,E))$ and $H_1(\Hh_k(G,E))$ are  generated  by  the indicator functions of bisections $B(\alpha, s(\alpha), \alpha; V)$ and $B(\alpha, g, \beta; V)$ for $\alpha, \beta\in E^k, g\in G_{s(\beta)}^{s(\alpha)}$ and for open subsets $V\subseteq Z(\alpha)$ and  $V\subseteq Z(\beta)$, respectively. Also,  $[\chi_{B(\alpha, g, \beta; V)}]=[\chi_{B(\alpha', g', \beta'; V')}]$ if and only if $V=V'$. 
Using the map \[j_k:\Hh_k(G,E)\hookrightarrow \Hh_{k+1}(G,E),\;\; j_k([\alpha, g, \beta; \beta x\mu])=[\alpha y, g|_x,\beta x; \beta x\mu],\] we obtain that 
$H_0(\Hh_k(G,E))\to H_0(\Hh_{k+1}(G,E))$ is given by
\[[\chi_{B(\alpha, u, \alpha ;V)}]\mapsto [\chi_{B(\alpha e_1, u, \alpha e_1; V)}]+[\chi_{B(\alpha e_3, v, \alpha e_3; V)}]\]
and
$H_1(\Hh_k(G,E))\to H_1(\Hh_{k+1}(G,E))$ is given by
\[[\chi_{B(\alpha, u, \beta; V)}]\mapsto [\chi_{B(\alpha e_1, u, \beta e_1; V)}]+[\chi_{B(\alpha e_3, v, \beta e_3; V)}].\]

We obtain $\ds H_i(\Hh(G,E))=\varinjlim_{k\to\infty}(C(uE^\infty,\ZZ), 2)\cong \ZZ[1/2]$ for $i=0,1$.

Now the groupoid $\Hh(G,E)$ is similar to $\Gg(G,E)\times_\rho\ZZ$, so we have a long exact sequence
\[0\longleftarrow H_0(\Gg(G,E))\longleftarrow H_0(\Hh(G,E))\stackrel{id-\rho_*}{\longleftarrow}H_0(\Hh(G,E))\longleftarrow H_1(\Gg(G,E))\]\[\hspace{100mm}\uparrow\]\[\hspace{30mm}0\longrightarrow H_2(\Gg(G,E))\longrightarrow H_1(\Hh(G,E))\stackrel{id-\rho_*}{\longrightarrow} H_1(\Hh(G,E))\]
where $\rho_*$ is the map induced by the action $\hat{\rho}: \ZZ\curvearrowright \Gg(G,E)\times_{\rho}\ZZ$ which takes $(\gamma, n)$ into $(\gamma, n+1)$. 

The map
$\rho_*:H_0(\Gg(G,E)\times_{\rho}\ZZ)\to H_0(\Gg(G,E)\times_{\rho}\ZZ)$ is given by
\[[\chi_{B(\alpha, s(\alpha), \alpha;Z(\alpha))\times\{0\}}]\mapsto [\chi_{B(\alpha, s(\alpha), \alpha;Z(\alpha))\times\{1\}}]\]
Consider $U=B(\alpha, u, \alpha e_1; Z(\alpha e_1))\times \{1\}\subseteq \Gg(G,E)\times_{\rho}\ZZ$ with
\[U^{-1}=B(\alpha e_1,u,\alpha;Z(\alpha))\times\{0\}.\]
Since
\[U^{-1}(B(\alpha, u, \alpha;Z(\alpha))\times\{1\})U=B(\alpha e_1, u, \alpha e_1;Z(\alpha e_1))\times \{0\},\]
it follows that in $H_0(\Gg(G,E)\times_{\rho}\ZZ)$ we have \[[\chi_{B(\alpha, u,\alpha;Z(\alpha))\times\{1\}}]=[\chi_{B(\alpha e_1, u, \alpha e_1;Z(\alpha e_1))\times \{0\}}]\] 
and
\[\rho_*([\chi_{B(\alpha, u, \alpha;Z(\alpha))\times\{0\}}])=[\chi_{B(\alpha e_1, u, \alpha e_1;Z(\alpha e_1))\times \{0\}}].\]
Hence $\rho_*$ on  $H_0(\Gg(G,E)\times_{\rho}\ZZ)\cong H_0(\Hh(G,E))\cong \ZZ[1/2]$  is multiplication by $1/2$.

Similarly,
$\rho_*:H_1(\Gg(G,E)\times_{\rho}\ZZ)\to H_1(\Gg(G,E)\times_{\rho}\ZZ)$ is given by
\[[\chi_{B(\alpha, g, \beta;Z(\beta))\times\{0\}}]\mapsto [\chi_{B(\alpha, g, \beta;Z(\beta))\times\{1\}}]\]
and if $U=B(\alpha, u, \alpha e_1; Z(\alpha e_1)\times \{1\}\subseteq \Gg(G,E)\times_{\rho}\ZZ$ we have
\[U^{-1}(B(\alpha, u, \beta;Z(\beta))\times\{1\})U=B(\alpha e_1, u, \beta e_1;Z(\beta e_1))\times \{0\}.\]
It follows that in $H_1(\Gg(G,E)\times_{\rho}\ZZ)$ we have \[[\chi_{B(\alpha, u,\beta;Z(\beta))\times\{1\}}]=[\chi_{B(\alpha e_1, u, \beta e_1;Z(\beta e_1))\times \{0\}}]\] 
and
\[\rho_*([\chi_{B(\alpha, u, \beta;Z(\beta))\times\{0\}}])=[\chi_{B(\alpha e_1, u, \beta e_1;Z(\beta e_1))\times \{0\}}],\]
so $\rho_*$ on $H_1(\Gg(G,E)\times_{\rho}\ZZ)\cong H_1(\Hh(G,E))\cong \ZZ[1/2]$ is also multiplication by $1/2$.

From the long exact sequence we obtain
\[H_0(\Gg(G,E))\cong \text{coker}(id-\rho_*),\; H_2(\Gg(G,E))\cong \ker (id-\rho_*)\]
and
\[0\to \text{coker}(id-\rho_*)\to  H_1(\Gg(G,E))\to \ker(id-\rho_*)\to 0.\]
It follows that
\[H_0(\Gg(G,E))\cong H_1(\Gg(G,E))\cong H_2(\Gg(G,E))\cong 0\]
and $H_q(\Gg(G,E)\cong 0$ for $q\ge 3$.


\bigskip

\begin{thebibliography}{0000}

\bigskip

\bibitem{AR} C. Anantharaman-Delaroche and  J. Renault, {\em  Amenable groupoids},  L'Enseignement Math\'ematique 36,
Gen\`eve, 2000. 

%\bibitem{BKQ} E. B\' edos, S. Kaliszewski and J. Quigg, {\em On Exel--Pardo algebras}, J. Operator Theory 78 (2017) 309--345.

\bibitem{B} K. S. Brown, {\em Cohomology of groups,}  Graduate
Texts in Mathematics, 87, Springer-Verlag, New York, 1994. %MR1324339
%\bibitem{BBGSW} N. Brownlowe, A. Buss, D. Gon\c calves, A. Sims and M. F. Whittaker, {\em $K$-theoretic duality for self-similar group actions on graphs}, in preparation.


\bibitem{CM} M. Crainic and I. Moerdijk, {\em A homology theory for \'etale groupoids}, J. Reine Angew.
Math. 521 (2000), 25--46. 

\bibitem{De1} V. Deaconu, {\em Generalized Cuntz-Krieger Algebras}, Proceedings of the AMS 124(1996) 3427--3435.

\bibitem{De} V. Deaconu, {\em Groupoid actions on  $C^*$-correspondences}, New York J. of Math. 24 (2018) 1020--1038.


%\bibitem{DGMW} R. J. Deeley, M. Goffeng, B. Mesland and M.F.  Whittaker, {\em Wieler solenoids, Cuntz-Pimsner algebras and $K$-theory}, Ergod. Th. \& Dynam. Sys 38 (2018), no. 8, 2942--2988.

\bibitem{E} D.G. Evans, {\em On the $K$-theory of higher rank graph $C^*$-algebras}, New York J. Math. 14 (2008), 1--31.
 
 \bibitem{EP} R. Exel and E. Pardo, {\em Self-Similar graphs: a unified treatment of Katsura and Nekrashevych algebras}, Adv. Math. 306(2017), 1046--1129.
 
 \bibitem{EPS} R. Exel, E. Pardo and C. Starling, {\em $C^*$-algebras of self-similar graphs over arbitrary graphs,}  arXiv:1807.01686 (2018).
 
% \bibitem{EanHR} R. Exel, A. an Huef and I. Raeburn, {\em Purely infinite simple $C^*$-algebras associated to integer dilation
%matrices}, Indiana Univ. Math. J. vol. 60 no. 3(2011) 1033--1058.
 
 \bibitem{FKPS} C. Farsi, A. Kumjian, D. Pask and A. Sims, {\em Ample groupoids: equivalence, homology, and Matui's HK conjecture},
M\"unster J. Math. 12 (2019), no. 2, 411--451.
 
  
 %\bibitem{LRR} M. Laca, I. Raeburn and J. Ramagge, {\em Phase transition on Exel crossed products associated to dilation matrices}, J. of Funct. Anal. 261(2011) 3633--3664.
 
 %\bibitem{LRRW14} M. Laca, I. Raeburn, J. Rammagge and M. F. Whittaker, {\em Equilibrium states on the Cuntz-Pimsner algebras of self-similar actions}, J. of Funct. Anal. 266 (2014) 6619--6661.
  
    \bibitem{LRRW} M. Laca, I. Raeburn, J. Ramagge and M. F. Whittaker, {\em Equilibrium states on operator algebras associated to self-similar actions of groupoids on graphs},  Adv. Math. 331 (2018), 268--325.
    
 \bibitem{L} H. Larki, {\em A dichotomy for simple self-similar graph algebras},  	arXiv:2005.05543.
    
    \bibitem{M}  H. Matui, {\em Homology and topological full groups of \'etale groupoids on totally disconnected spaces}, Proc. Lond. Math. Soc. (3) 104 (2012), no. 1, 27--56.
       

 
 \bibitem{N} V. Nekrashevych, {\em Self-similar groups}, Math. Surveys Monogr. 117, AMS Providence 2005.

\bibitem{N1} V. Nekrashevych, {\em $C^*$-algebras and self-similar groups}, J. Reine Angew. Math. 630 (2009) 59--123.

\bibitem{O} E. Ortega, {\em Homology of the Katsura-Exel-Pardo groupoid}, J. Noncommut. Geom. 14(2020) no. 3, 913--935.

\bibitem{OS} E. Ortega and A. Sanchez, {\em The homology of the groupoid of the self-similar dihedral group}, arXiv:2011.02554v1.
   


\bibitem{Re} J. Renault, {\em A groupoid approach to $C^*$-algebras}, Lecture Notes in Mathematics, 793, Springer, Berlin, 1980. 

\bibitem{S} E. Scarparo, {\em Homology of odometers}, Ergod. Th. \& Dynam. Sys (2020), 40, 2541--2551.


\bibitem{T} J. L. Tu, {\em Groupoid cohomology and extensions}, Trans. Amer. Math. Soc. 358 (2006), no. 11, 4721--4747. 

\bibitem{Y} I. Yi, {\em Homology and Matui's HK conjecture for groupoids on one-dimensional solenoids}, Bull. Aust. Math. Soc 101(2020), 105--117.
\end{thebibliography}
\end{document}